\documentclass[10pt,a4paper]{article}
\usepackage{amsmath,amsfonts,amsthm,amsopn,color,amssymb,graphics,graphicx}
\newcommand{\R}{\mathbb{R}}
\newcommand{\G}{\Gamma}
\newcommand{\eps}{\varepsilon}
\newtheorem{theorem}{Theorem}[section]
\newtheorem{lemma}[theorem]{Lemma}

\newtheorem{proposition}[theorem]{Proposition}

\newtheorem{corollary}[theorem]{Corollary}
\renewenvironment{proof}[1][Proof.]{\noindent\textbf{#1} }{$\hfill\square$\vspace{0.2 cm}\\}
\numberwithin{equation}{section}
\begin{document}

\title{{\bf{Stable determination of \\ surface impedance on a rough obstacle \\ by far field data} \thanks{Work supported in part by PRIN 20089PWTPS}}}
\author{Giovanni Alessandrini\footnote{Universit\`a degli Studi di Trieste, Italy, e-mail: \text {alessang@units.it}}  \and
Eva Sincich\footnote{Universit\`a degli Studi di Trieste, Italy,  e-mail: \text{esincich@units.it}} \and Sergio Vessella\footnote{Universit\`a degli Studi di Firenze, Italy,  e-mail: \text{sergio.vessella@dmd.unifi.it}}
}

\date{}

\maketitle

\begin{abstract}
\noindent We treat the stability of determining the boundary impedance of an obstacle by scattering data, with a single incident field. A previous result by Sincich (SIAM J. Math. Anal. 38, (2006), 434-451) showed a log stability when the boundary of the obstacle is assumed to be $C^{1,1}$-smooth. We prove that, when the obstacle boundary is merely Lipschitz, a log-log type stability still holds.
\end{abstract}

{\small{\bf Keywords:} Inverse scattering problem, impedance boundary condition, stability. \ }

{\small{\bf 2000 Mathematics Subject Classification }: 35R30, 35R25,
31B20.}

\section{Introduction}
In this note we consider the stability of the determination of an
impedance coefficient on an inaccessible boundary from scattering
data arising from a single incident wave.\\
More precisely we consider
\begin{equation}
   \label{Sc}
\left\{
\begin{array}{lr}
     \Delta u+k^{2}u=0, &\hbox{in } \mathbb{R}^{3}\setminus \overline{D},\\
      \frac{\partial u}{\partial \nu }+i\lambda u=0,& \hbox{on }\partial D,\\
\end{array}
\right.
\end{equation}
where $u=u^{s}+\exp (ikx\cdot \omega )$, $\omega \in \mathbb{S}^{2}$
and the scattered field $u^{s}$ satisfies the usual
\textit{Sommerfeld
radiation condition}%
\begin{equation}
\label{Sommerfield}
\lim_{r\rightarrow \infty }r\left( \frac{\partial u^{s}}{\partial r}%
-iku^{s}\right) =0\text{, \ \ }r=\left\vert x\right\vert.
\
\end{equation}
The scattered field has a well-known asymptotic behavior at $\infty
$
\begin{equation} \label{Farfield}
u^{s}(x)=\frac{\exp (ikr)}{r}\left\{ u_{\infty }\left(
\widehat{x}\right) +O\left( \frac{1}{r}\right) \right\} \
\end{equation}
as $r$ tends to $\infty$, uniformly with respect to $\widehat{x}=\frac{x}{%
\left\vert x\right\vert }$ and where $u_{\infty }$ is the so-called
far field pattern of the scattered wave (see \cite{CK}).\\

We deal with the inverse scattering problem of determining $\lambda=\lambda(x)$
when the far field $u_{\infty }$ is known for a given incident
direction $\omega$.\\

The stability for this problem was treated in \cite{Eva}, and
we refer to the bibliography therein for further references.\\

The result in \cite{Eva} can be summarized as follows.
If the unknown impedance $\lambda$ is Lipschitz and the boundary of the scatterer $D$ is $C^{1,1}$, then $\lambda$ depends on $u_{\infty}$ with a modulus of continuity of logarithmic type (with a single log !).\\

The purpose of this note is to investigate the stability when the regularity assumption on the scatterer is relaxed to Lipschitz. The requirement of the $C^{1,1}$ regularity at the boundary was mainly due to the fact that such regularity is needed in order to obtain a doubling inequality at the boundary for $|u|^2$ \cite[Theorem 4.6]{Eva}. We refer to \cite{AE} for a thorough discussion of this topic.
In brief, such a doubling inequality is needed because the impedance can be obtained from the total field $u$ (which is uniquely, and stably, determined by the far field pattern) by computing
\[
\lambda =\frac{i}{u}\frac{\partial u}{\partial \nu }\text{.}
\]
Consequently, it is necessary, in order to estimate $\lambda $, to have a
control on the rate of vanishing of $u$ on $\partial D$. The doubling
inequality implies that such a rate is at most algebraic in the $L^{2}$
-average sense, that is
\[
\int_{\Delta _{r}(x_{0})}\left\vert u\right\vert ^{2}\geq \frac{1}{K}r^{K}
\]%
for some constant $K>0$. Here $\Delta _{r}(x_{0})=B_{r}(x_{0})\cap \partial D
$ and $B_{r}(x_{0})$ is the ball of radius $r$ centered at $x_{0}\in
\partial D$. From this bound on the vanishing rate of $|u|^{2}$ the
stability for $\lambda $ with a \emph{single} log follows. Note that in
fact, in \cite{Eva} a slightly different route is followed, which involves the
notion of Muckenhoupt weights.

It is an open problem if the doubling inequality at the boundary holds true
when we relax the regularity assumption on $\partial D$ (see again \cite{AE}). Nevertheless, it is expected that the order of vanishing of 
$u_{|\partial D}$ can be still controlled somehow. In fact we shall prove
that, assuming $\partial D$ Lipschitz, $u_{|\partial D}$ has a rate of
vanishing controlled in an exponential fashion as follows
\begin{equation}\label{expweight}
\int_{\Delta _{r}(x_{0})}\left\vert u\right\vert ^{2}\geq
e^{-Kr^{-K}}\text{.}
\end{equation}
With the aid of such estimate we are able to prove a stability result of a
weaker type, namely with a log-log type modulus of continuity.\\
We remark that this phenomenon confirms a scheme that was already known for  different, but related inverse boundary value problems \cite{ABRV}. In that paper, inverse problems with unknown boundary for elliptic equations were treated and, on the unknown boundary a condition of Dirichlet or Neumann type was prescribed. It was shown that, in case that the unknown boundary is smooth enough to have a doubling inequality at the boundary, then the stability was of log type. Instead with weaker regularity assumptions only a log-log type estimate could be obtained. See for instance the proof of \cite[Theorem 2.1]{ABRV}. This scheme has also been confirmed with different inverse boundary problems \cite{MoRos, Ballerini}.

Let us mention also that, recently, a similar task to the present one was attempted in the preprint \cite{BCJ} by Bellassoued, Choulli and Jbalia. The main result in that paper also consists of a \textit{log-log} stability, but \textit{of local type only}. That is, the stability is proved only when the error on the unknown impedance is a-priori known to be small. Furthermore the stated regularity assumptions on the boundary have some unclear aspects. 

The plan of the paper is as follows. In Section \ref{mainass} we introduce notation, the main assumptions and state the main stability result, Theorem \ref{stabl}. Section \ref{directp} collects more or less well-known estimates on the direct problem \eqref{Sc}, we conclude it by stating a rather more delicate result, Theorem \ref{h1}, in which the boundary trace of the field $u$ is estimated in $H^1$. The proof is contained in the following
Section \ref{rellich}. It relies on the well-known Rellich's identity, the main theorem of this section being Theorem \ref{localbnds}. Section \ref{scattf} explores the stability of the determination of the scattered field in terms of the far field data. We review some known results, and we present some new ones, which rely on the analysis developed in the previous Section \ref{rellich}, see in particular Corollaries \ref{hmenouno} and \ref{hmenounmezzo}, where the  stability for negative norms of the normal derivative of $u$ is obtained. The occurrence of negative norms appears to be necessary here because, being the boundary only Lipschitz, the normal derivative of $u$ may be nonsmooth. In Section \ref{lps} we obtain,  Theorem \ref{thmlps}, the local lower bound on the vanishing rate of $u$ on $\partial D$ announced above in \eqref{expweight}. This is achieved though a quantitative estimate of unique continuation which is by now well-known as a Lipschitz Estimate of Propagation of Smallness, see for instance \cite{ABRV, MoRos}. In Section \ref{secweighted} we present an interpolation inequality between a weighted $L^1$-norm and a H\"{o}lder norm, when the weight satisfies a bound on its vanishing rate of the type \eqref{expweight}, Proposition \ref{weighted}. We finish the proof of Theorem \ref{stabl} in the final Section \ref{final}.

\section{Main assumptions and results}\label{mainass}

\subsection{Main hypotheses and notation}
\emph{\bf Assumptions on the domain.}\\
We shall assume throughout that $D$ is a bounded domain in $\mathbb{R}^3$, that is, for a given $d>0$, we require $\mbox{diam}D\le d$. Also we require 
that $\mathbb{R}^{3}\setminus \overline{D}$ is a connected set, and that the boundary $\partial D$ is Lipschitz
with
constants $r_0, M$.
More precisely,
for every $x_0 \in \partial D$, there exists a rigid transformation of
coordinates under which,
\begin{eqnarray}\label{gamma}
D \cap B_{r_0}(x_0)=\{(x',x_3): x_3>\gamma(x')\}\ ,
\end{eqnarray}
where $x \in \R^3,\ x=(x',x_3)$, with $x'  \in \R^2 , \ x_3 \in
\mathbb{R}$ and
$$\gamma :B'_{r_0}(x_0)\subset \R^2 \rightarrow \mathbb{R}$$
satisfying $\gamma(0)=0$ and $$\|\gamma\|_{C^{0,1}( B'_{r_0}(x_0))}\le
Mr_0,$$ where we denote by
$$\|\gamma\|_{C^{0,1}( B'_{r_0}(x_0))}=\|\gamma\|_{L^{\infty}(
B'_{r_0}(x_0))} +\ r_0\!\!\!\!\!\!\!\!\!\sup_{\substack {x,y  \in
B'_{r_0}(z_0)\\x\ne y }}
\frac{|\gamma (x)-\gamma (y)|}{|x-y|} _{\ \ \  } $$
and $B'_{r_0}(x_0)$ denotes a ball in $\R^2$.

For the sake of simplicity we shall assume that $0\in D$.

Fixed $R>d$, $\rho\in(0,r_0)$ and $x_0\in\partial D$, let us define the
following sets

\begin{eqnarray}
&&D^{+}=\R^3\setminus\overline{D},\\
&&D_R^{+}=B_R(0)\cap D^{+}, \\
&&E_{\rho}=\{x\in \R^n | \ 0<\mbox{dist}(x,\overline{D})<\rho) \},\\
&&D_{\rho}=\overline{D}\cup E_{\rho},\\
&&\G_{\rho}(x_0)= B_{\rho}(x_0)\setminus \overline{D}\label{neigh}
,\\
&&\Delta_{\rho}(x_0)=B_{\rho}(x_0)\cap \partial
D.\label{neighb}
\end{eqnarray}

\emph{\bf A priori information on the impedance term.}\\
Given $\lambda_0>0$, we assume that 
the impedance coefficient 
$\lambda$ be\-longs to 
$C^{0,1}(\partial D,\R)$ and is such that
\begin{eqnarray}\label{limdalbasso}
\lambda(x)\ge\lambda_0>0
\end{eqnarray}
 for every $x\in \partial D$. Moreover we assume that, for a given constant
$\Lambda>0$, we
have that
\begin{eqnarray}\label{L}
\|\lambda\|_{C^{0,1}(\partial D)}\le \Lambda .
\end{eqnarray}

From now on we shall refer to the \emph{a priori data} as to the
following set of quantities,
$d, r_0, M,\lambda_0, \Lambda, k$.\\
In the sequel we shall denote with $\eta(t)$ a positive increasing concave
function defined on $(0, +\infty)$, that satisfies
\begin{eqnarray}\label{eta}
&&\eta(t)\le C(\log (t))^{-\vartheta},\ \ \ \mbox{for
every}\ \ 0<t<1 \ ,
\end{eqnarray}
where
$C>0,\vartheta>0$ are constants depending on the \emph{a priori
data} only.

\subsection{The main result}
\begin{theorem}[\textbf{Stability for $\lambda$}]\label{stabl} Let $\lambda_1, \lambda_2$ satisfy \eqref{limdalbasso}, \eqref{L}.
Let $u_i, \ i=1,2$, be the weak solutions to the problem \eqref{Sc}
with $\lambda=\lambda_i$ respectively and let $u_{i,\infty}$ be their
respective far field patterns.
If for some
 $\varepsilon> 0$, we have
\begin{eqnarray}\label{farf}
\|u_{1,\infty}-u_{2,\infty}\|_{L^2(\partial B_1(0))}\le \varepsilon,
\end{eqnarray}
then
\begin{eqnarray}\label{aim}
\|\lambda_1-\lambda_2\|_{L^{\infty}(\partial D)}\le \eta \circ \eta(\eps),
\end{eqnarray}
where $\eta$ is given by \eqref{eta}.
\end{theorem}

\section{The direct problem}\label{directp}

Let us introduce the following space
$$H^1_{\mbox{loc}}(D^+)=\{v\in \mathcal{D}^{'}(D^+):\
v|_{D_{R}^+}\in H^1(D_{R}^+),\ \mbox{for every}\ R>0\ \mbox{s.t.}\
\overline{D}\subset B_{R}(0)\}$$
where $\mathcal{D}^{'}$ is the space of distribution on $D^+$.

A weak solution to the problem \eqref{Sc} is a function
$u=\exp{(ik\omega\cdot x)}+u^s$, where $u^s\in
H^1_{\mbox{loc}}(D^+)$ is a weak solution to the problem
\begin{eqnarray}\label{radiating}
\left\{
\begin{array}
{lll}
\Delta u^s +k^2u^s=0,& \mbox{in $D^+$},
\\
\dfrac{\partial u^s}{\partial \nu} +
i\lambda(x)u^s=-\dfrac{\partial}{\partial \nu}\exp{(ik\omega\cdot x)} -
i\lambda(x)\exp{(ik\omega\cdot
x)}, & \mbox{on $\partial D$},
\\
\lim_{r\rightarrow\infty}r\left(\dfrac{\partial u^{s}}{\partial
r}(r\hat{x})-iku^s(r\hat{x}) \right)=0, & \mbox{uniformly in}\
\hat{x}.
\end{array}
\right.
\end{eqnarray}
where $\nu$ is the inward unit normal to $D$.

\begin{lemma}[\textbf{Well-posedness}]\label{benposto}
The problem \eqref{radiating} has one and only one weak solution $u^s$.
Moreover, for every $R>d$, there exists a constant $C_R>0$ depending
on the \emph{a priori data} and on $R$ only, such that the following bound
holds
\begin{eqnarray}\label{dipcont}
\|u^s\|_{H^1(D_R^+)}\le C_R\ .
\end{eqnarray}
\end{lemma}
\begin{proof}
For the proof we refer to Lemma 3.1 in \cite{Eva}, see also \cite{CCM} for a previous related result.
\end{proof}
\begin{theorem}[\textbf{$C^{\alpha}$ regularity at the
boundary}]\label{regolarita}

Let $u$ be the weak solution to \eqref{Sc}, then there exists a
constant $\alpha,\ 0<\alpha<1$, such that for every $R>d$  $u\in C^{\alpha}\left(\overline{D_{R}^+}\right)$. Moreover, there exists a
constant $C_{R}>0$ depending on the \emph{a priori data}, on
$R$ only, such that
\begin{eqnarray}\label{ciunoalfa}
\left \Vert u\right \Vert _{C^{\alpha }\left( \overline{D_{R}^{+}}\right) }\le C_R\  .
\end{eqnarray}
\end{theorem}
\begin{proof}
This is a more or less standard regularity estimate up to the boundary. The Moser iteration technique fits to this task. Details can be found in \cite[Lemma 3.3]{EvaTesi}. Note also that the arguments used there only require
the Lipschitz regularity of $\partial D$.
\end{proof}

\begin{corollary}[\textbf{Lower bound}]\label{lb}
There exists a radius $R_0>0$ depending on the a priori data only such that
\begin{eqnarray}\label{lbi}
|u(x)|>\frac{1}{2} \ \ \mbox{for any}\ \ |x|>R_0\ .
\end{eqnarray}
\end{corollary}
\begin{proof}
The proof relies on the same arguments discussed in \cite[Corollary 3.3]{Eva}.
\end{proof}

\begin{theorem}[\textbf{$H^1(\partial D)$}]\label{h1}
Let $u$ be as in Theorem \ref{regolarita}, then
\begin{eqnarray}
\|u\|_{H^1(\partial D)}\le C \ .
\end{eqnarray}
\end{theorem}
The proof shall be given in the next section.

\section{Estimates at the boundary.}\label{rellich}

We begin by recalling the well-known fact that there exists $\rho_{0} \in \left( 0,r_{0}\right] $ depending only on $
M,r_{0},d$ and $k$ only such that, for every $\rho \in \left( 0,\rho _{0}
\right] $ the following coerciveness condition is satisfied
\begin{eqnarray}\label{coerc}
\int_{E_{\rho
}}\left( \left \vert \nabla v\right \vert ^{2}-k^{2}\left \vert v\right \vert
^{2}\right) \geq \frac{1}{2}\left \Vert v\right \Vert _{H^{1}\left( E_{\rho
}\right) }^{2} \text{ in }E_{\rho }
\end{eqnarray}

for all $v\in H_{0}^{1}\left( E_{\rho }\right) $, see for instance \cite[Lemma 8.4]{gt}.
In what follows we shall fix $\rho\in \left( 0,\rho _{0}\right] $.\\

\begin{theorem}\label{localbnds}
Let $v\in H^{1}\left( E_{\rho }\right) $ be a solution to
\begin{eqnarray}
\Delta v+k^{2}v=0  \text{ in }E_{\rho }\text{.}
\
\end{eqnarray}
If its trace $v_{|\partial D}\in H^{1}\left( \partial D\right) $ then $\frac{
\partial v}{\partial \nu }_{|\partial D}\in L^{2}\left( \partial D\right) $
and we have
\begin{eqnarray}\label{normal}
\left \Vert \frac{\partial v}{\partial \nu }\right \Vert _{L^{2}\left(
\partial D\right) }^{2}\leq C\left( \left \Vert \nabla _{T}v\right \Vert
_{L^{2}\left( \partial D\right) }^{2}+\left \Vert v\right \Vert _{H^{1}\left(
E_{\rho }\right) }^{2}\right) \text{.}
\end{eqnarray}
Conversely, if $\frac{\partial v}{\partial \nu }_{|\partial D}\in
L^{2}\left( \partial D\right) $ then $v_{|\partial D}\in H^{1}\left(
\partial D\right) $ and
\begin{eqnarray}\label{tangential}
\left \Vert \nabla _{T}v\right \Vert _{L^{2}\left( \partial D\right) }^{2}\leq
C\left( \left \Vert \frac{\partial v}{\partial \nu }\right \Vert _{L^{2}\left(
\partial D\right) }^{2}+\left \Vert v\right \Vert _{H^{1}\left( E_{\rho
}\right) }^{2}\right) \text{.}
\end{eqnarray}
Here $\nabla _{T}v$ denotes the tangential gradient of $v$ on $\partial D$
and $C$ depends on $M,r_{0},d,\rho $ and $k$ only.
\end{theorem}

\begin{proof} A priori inequalities of this sort were first proven by Payne and Weinberger \cite{PW1}, \cite{PW2}. A proof when $k=0$ and with Lipschitz boundary is due to Jerison and Kenig \cite{JK} . The underlying tool for such estimates relies on the celebrated Rellich's identity \cite{Re}. The adaptation to the case of the Helmholtz equation ($k\neq0$) can be obtained as follows.\\
Recalling the local graph representation of $\partial D$ introduced in \eqref{gamma}, let us define, for every $r\in \left( 0,2\rho_{0}\right] $, the half cylinder
\begin{eqnarray}
C_{r}^{+}=\left \{ x=\left( x^{\prime },x_{3}\right) \in \mathbb{R}
^{3}||x^{\prime }|<r\text{, }\gamma \left( x^{\prime }\right)
<x_{3}<Mr\right \} \ ,
\end{eqnarray}
and let us denote
\begin{eqnarray}
\Delta_{r}=\left \{ x=\left( x^{\prime },x_{3}\right) \in \mathbb{R}
^{3}||x^{\prime }|<r\text{, }x_3=\gamma \left( x^{\prime }\right)
\right \} \ .
\end{eqnarray}

First we prove local estimates of the following form
\begin{eqnarray}\label{Nloc}
\int_{\Delta _{r}}\left \vert \frac{\partial v}{\partial \nu }\right \vert
^{2}\leq C\left( \int_{\Delta _{2r}}\left \vert \nabla _{T}v\right \vert
^{2}+\int_{C_{2r}^{+}}\left \vert \nabla v\right \vert ^{2}+\left \vert
v\right \vert ^{2}\right)
\end{eqnarray}

\begin{eqnarray}\label{Dloc}
\int_{\Delta _{r}}\left \vert \nabla _{T}v\right \vert ^{2}\leq C\left(
\int_{\Delta _{2r}}\left \vert \frac{\partial v}{\partial \nu }\right \vert
^{2}+\int_{C_{2r}^{+}}\left \vert \nabla v\right \vert ^{2}+\left \vert
v\right \vert ^{2}\right)
\end{eqnarray}
for every $r\in \left( 0,\rho_{0}\right] $ and $C$ depends on $M,r_{0},r$ and $k
$ only.
These estimates are proven in the case $k=0$ in \cite[Proposition 5.1]{AMoRos}. The extension to the case $k\neq0$ can be either obtained by a straightforward adaptation of the argument in \cite[Section 5]{AMoRos}, or else by introducing an auxiliary variable $t$ and posing $V(x,t)=v(x)e^{kt}$ which is harmonic in $C_{2r}^{+}\times \mathbb{R} $. Writing down the analogs in higher dimension of \eqref{Nloc} and \eqref{Dloc} for $V$, we readily deduce the above stated inequalities for $v$. The final step consists of covering $\partial D$ with neighbourhoods $\Delta_r$, with $r$ small enough so that the corresponding half cylinders $C_{2r}^{+}$ remain within $E_{\rho}$. The number of the neighbourhoods $\Delta_r$ needed to cover $\partial D$ can be estimated in terms of $M,r_{0},d,\rho $, see \cite[Proof of Prop. 3.3]{AMoRos}.
\end{proof}
\begin{proof}[Proof of Theorem \ref{h1}]
Immediate consequence of \eqref{tangential} in combination with the impedance condition in \eqref{Sc} and Lemma \ref{benposto}.
\end{proof}

We shall denote by $\vartheta \in C^{\infty }\left( \mathbb{R}^{n}\right) $ a 
mollification of the characteristic function $\chi_{D_{\rho /2}}$ such that
$0\leq \vartheta \leq 1$, $\vartheta \equiv 1$ in $D_{\rho /3}$, $\vartheta \equiv 0$
in $\mathbb{R}^{n}\setminus D_{2\rho /3}$, $\left \vert D^{m}\vartheta
\right \vert \leq C_{m}\rho ^{-m}$ in $\mathbb{R}^{n}$, where $C_{m}$ depends
on $m$, $d$ and $r_{0}$.

\begin{lemma} \label{stimaduale}
Let $u\in H^{1}\left( E_{\rho }\right) \cap C^{0}\left( \overline{E_{\rho }}
\right) $,  be a solution to%
\begin{eqnarray}
\Delta u+k^{2}u=0\text{ in }E_{\rho }\text{.}
\end{eqnarray}
We have
\begin{eqnarray}\label{H^{-1}}
\left \Vert \frac{\partial u}{\partial \nu }\right \Vert _{H^{-1}\left(
\partial D\right) }\leq C\left \Vert u\right \Vert _{L^{\infty }\left( E_{\rho
}\right) }
\end{eqnarray}
where $C$ depends on $M,r_{0},d,\rho $ and $k$ only.
\end{lemma}
\begin{proof} In view of \eqref{coerc}, for any $\zeta \in H^{1}\left( \partial D\right) $ we can
consider the unique solution $\varphi \in H^{1}\left( E_{\rho }\right) $ to
the Dirichlet problem
\begin{eqnarray}
\left \{
\begin{array}{c}
\Delta \varphi +k^{2}\varphi =0\text{ in }E_{\rho } \ ,\\
\varphi =\zeta \text{ on }\partial D \ , \\
\varphi =0\text{ on }\partial E_{\rho }\setminus \partial D \ .
\end{array}
\right.
\end{eqnarray}
Moreover we have
\begin{eqnarray}\label{estim}
\left \Vert \varphi \right \Vert _{H^{1}\left( E_{\rho }\right) }\leq
C\left \Vert \zeta \right \Vert _{H^{1/2}\left( \partial D\right) }
\end{eqnarray}
with $C>0$ only depending on $M,r_{0},d,\rho $ and $k$. Let $\vartheta $ the
previously introduced cutoff function and denote $\psi =\vartheta \varphi $.
The Green's identity gives
\begin{eqnarray}
\int_{\partial D}\psi \frac{\partial u}{\partial \nu }-\int_{\partial D}u%
\frac{\partial \psi }{\partial \nu }=\int_{E_{\rho }}u\left( \Delta \psi
+k^{2}\psi \right)
\end{eqnarray}
hence
\begin{eqnarray}
\left \vert \int_{\partial D}\zeta \frac{\partial u}{\partial \nu }%
\right \vert \leq \left \vert \int_{\partial D}u\frac{\partial \varphi }{%
\partial \nu }\right \vert +\left \vert \int_{E_{\rho }}u\left( 2\nabla \vartheta
\cdot \nabla \varphi +\left( \Delta \vartheta +k^{2}\vartheta \right) \varphi
\right) \right \vert
\end{eqnarray}
applying \eqref{normal} to $\varphi $ and taking into account \eqref{estim} we get
\begin{eqnarray}
&\left \vert \int_{\partial D}\zeta \frac{\partial u}{\partial \nu }
\right \vert \leq  &\nonumber \\ &\leq C\left( \left \Vert u\right \Vert _{L^{\infty }\left( E_{\rho
}\right) }\left \Vert \zeta \right \Vert _{H^{1}\left( \partial D\right)
}+\left \Vert u\right \Vert _{L^{\infty }\left( E_{\rho }\right) }\left \Vert
\varphi \right \Vert _{H^{1}\left( E_{\rho }\right) }\right)\leq &\nonumber \\ &\leq C\left \Vert
u\right \Vert _{L^{\infty }\left( E_{\rho }\right) }\left \Vert \zeta
\right \Vert _{H^{1}\left( \partial D\right) } &
\end{eqnarray}
and the thesis follows by duality.
\end{proof}

\section{Stability for the scattered field} \label{scattf}

\begin{lemma}[\textbf{From the far field to the near
field}]\label{farnear}
Let $u_i,u_{i,\infty},\ i=1,2$, be as in Theorem \ref{stabl}. Suppose
that, for some $\varepsilon,\ 0<\varepsilon<1$,
\eqref{farf} holds, then there exist a radius $R_1>0$ and a constant
$C>0$, depending
on the \emph{a priori data} only, such that
\begin{eqnarray}\label{nearfield}
\|u_1-u_2\|_{L^2(B_{R_1+1}(0)\setminus B_{R_1}(0))}\le C
{\varepsilon}^{\alpha(\varepsilon)},
\end{eqnarray}
where $\alpha(\varepsilon)$ is defined as
\begin{eqnarray}\label{lam}
\alpha(t)=\frac{1}{1+\log(\log({t}^{-1})+e)}\ .
\end{eqnarray}
\end{lemma}
\begin{proof}
For the proof we refer to Lemma 4.1 in \cite{Eva}, which is based on the stability results for the near field achieved by Isakov in \cite{I} and  further developed by Bushuyev in \cite{Bu}.
\end{proof}

\begin{theorem}[\textbf{Stability at the boundary}]\label{stabpc}
Let $u_i,u_{i,\infty},\ i=1,2$, be as in Theorem \ref{stabl}. We have
that,
if for some $\eps> 0$, \eqref{farf} holds, then
\begin{eqnarray}\label{stabilitaalbordo}
\|u_1-u_2\|_{L^{\infty}(\partial D)}\le \eta(\eps)\ ,
\end{eqnarray}
where $\eta$ is given by \eqref{eta}, with a constant $C>0$ depending
on the \emph{a priori data} only.
\end{theorem}
\begin{proof}
The arguments in \cite{Eva} need few adjustments. The main additional tool is a global estimate of propagation of smallness \cite[Theorem 5.3]{ARonRosV} which enables to achieve
\begin{eqnarray}
\|u_1 -u_2 \|_{L^2(D_R^+)}\le \eta(\varepsilon) \ .
\end{eqnarray}
The $C^\alpha$ bound obtained in Theorem \ref{regolarita} allows, by interpolation, to conclude.

\end{proof}

\begin{corollary}[\textbf{$H^{-1}$ bound}]\label{hmenouno}
Under the same hypothesis of Theorem \ref{regolarita} we have that

\begin{eqnarray}
\left \|\dfrac{\partial u_1}{\partial \nu}- \dfrac{\partial u_2}{\partial \nu}\right \|_{H^{-1}(\partial D)}\le \eta(\eps)\ .
\end{eqnarray}
\end{corollary}
\begin{proof}
This is an immediate consequence of \eqref{stabilitaalbordo} and of Lemma \ref{stimaduale}.
 \end{proof}
\begin{corollary}\label{hmenounmezzo}
Under the same hypothesis of Theorem \ref{regolarita} we have that

\begin{eqnarray}
\left \|\dfrac{\partial u_1}{\partial \nu}- \dfrac{\partial u_2}{\partial \nu}\right \|_{H^{-\frac{1}{2}}(\partial D)}\le \eta(\eps)\ .
\end{eqnarray}
\end{corollary}

\begin{proof}
The result follows by interpolation with the aid of the impedance condition in \eqref{Sc} and of Theorem \ref{regolarita}.
\end{proof}

\section{The estimate of Lipschitz propagation \\ of smallness} \label{lps}

Let us denote $r_1= \min\left\{r_0, \frac{1}{2}\right\}$.
\begin{theorem}\label{thmlps} Let $u$ be the weak solution to \eqref{Sc}.
For every $r$, $0<r<r_1$, and for every $x_0\in \partial D$ we have that
\begin{eqnarray} \label{eq:lps}
\int_{\Delta_r(x_0)}|u|^2\ge \exp(-Kr^{-K})\ \
\end{eqnarray}
where $K>0$ only  depends on the a priori data.
\end{theorem}

The proof shall stem from the two Lemmas below.

\begin{lemma}Let $u$ be the weak solution to \eqref{Sc}.
For every $r$, $0<r<r_0$ and for every $x_0\in \partial D$ we have that
\begin{eqnarray}
\left(\int_{\Delta_{r}(x_0)}|u|^2\right)^{\beta}\ge C \int_{\Gamma_{\frac{r}{2}}(x_0)}|u|^2\ \
\end{eqnarray}
where $C>0,\ 0<\beta<\frac{1}{2}$ are constants depending on the a priori data only.
\end{lemma}
\begin{proof} By the arguments in \cite[Theorem 1.7]{ARonRosV} concerning the well-known local estimate for the Cauchy problem and by the a priori bounds achieved in Theorem \ref{regolarita} and Theorem \ref{h1} we get that for any $0<r<r_0$ and for any $x_0\in \partial D$
\begin{eqnarray}
&\|u\|_{L^2(\Gamma_{\frac{r}{2}}(x_0))}\le &\nonumber \\ &\leq C \left(\|u\|_{H^{\frac{1}{2}}(\Delta_{r}(x_0))} +  \|\partial_{\nu}u\|_{H^{-\frac{1}{2}}(\Delta_{r}(x_0))} \right )^{\delta} \left( {\|u \|_{L^2(\Gamma_r(x_0))}}\right)^{1-\delta} & \end{eqnarray}
where $C>0, 0<\delta<1$ are constants depending on the a priori data only.
Note that the last factor $\|u \|_{L^2(\Gamma_r(x_0))}^{1-\delta}$ is bounded by a constant in view of Theorem \ref{regolarita}.

We recall that for any $0<r<r_0$ and for any $x_0\in \partial D$ the following interpolation inequality holds
\begin{eqnarray}
\|u\|_{H^{\frac{1}{2}}(\Delta_{r}(x_0))}\le C \|u\|^{\frac{1}{2}}_{L^{{2}}(\Delta_{r}(x_0))}\|u\|^{\frac{1}{2}}_{H^{{1}}(\Delta_{r}(x_0))}
\end{eqnarray}
where $C>0$ depends on the a priori data only.

Finally by the impedance condition on $\partial D$, by the above interpolation inequality and by the a priori bound in Theorem \ref{h1}, the theorem follows with $\beta=\frac{\delta}{2}$.
\end{proof}

\begin{lemma}Let $u$ be as above.
For every $r$, $0<r<\frac{r_0}{2}$ and for every $x_0\in \partial D$ we have that
\begin{eqnarray}
\int_{\Gamma_r(x_0)}|u|^2\ge C\exp(-k_1r^{-k_2})\ \
\end{eqnarray}
where $C,k_1,k_2>0$ are constants depending on the a priori data only.

\end{lemma}
\begin{proof}
For the proof we mainly refer to \cite[Proposition 3.1]{MoRos}, where the authors achieved a Lipschitz propagation of smallness result for the $L^2$ norm of Jacobian matrix of the solution to the Lam\'{e} system. Nevertheless, the same arguments go through for the $L^2$ norm of the solution to the Helmholtz equation as well.
We provide below a sketch of the proof.

Let $\bar{x}\in \Gamma_{r}(x_0)$ be such that $B_{\frac{r}{8}}(\bar{x})\subset \Gamma_{r}(x_0)$. Now, by adapting the techniques developed in
\cite[Proposition 3.1]{MoRos}, which are mostly based on a standard propagation of smallness and the iterated use of the three spheres inequality along a chain of balls centered on the axis of a cone and with increasing radii (see also \cite[Section E.3]{ADiB}), we find that
\begin{eqnarray}
\int_{B_{\frac{r}{8}}(\bar{x})}|u|^2\ge C\exp(-k_1r^{-k_2})\int_{D^+_{2R_0}}|u|^2
\end{eqnarray}
where $R_0>0$ is the radius introduced in Corollary \ref{lb}.

By combining the lower bound stated in \eqref{lbi} and the obvious inequality \\$\int_{\Gamma_{\frac{r}{2}}(x_0)}|u|^2\ge\int_{B_{\frac{r}{8}}(\bar{x})}|u|^2$ we obtain the desired estimate.
\end{proof}

\begin{proof}[Proof of Theorem \ref{thmlps}]
From the above Lemmas we deduce
\begin{eqnarray}
\int_{\Delta_r(x_0)}|u|^2\ge C^{\frac{2}{\beta}}\exp(-\frac{k_1}{\beta}r^{-k_2})\ .
\end{eqnarray}
It is an elementary fact that we may find $K>0$, only depending on $C, k_1, k_2, \beta$ such that \eqref{eq:lps} follows.
\end{proof}

\section{A weighted interpolation inequality}\label{secweighted}
\begin{proposition}\label{weighted}
Given $ M, K >0$, let $w\geq 0$ be a measurable function on $\partial D$ satisfying the conditions
\begin{eqnarray}\label{infinito}
\|w\|_{L^{\infty}(\partial D)} \leq M
\end{eqnarray}
and
\begin{eqnarray}\label{zeriexp2}
  \|w\|_{L^{2}(\Delta_r(x))} \geq \exp(- Kr^{-K}) \ \text{for every } x \in \partial D\text{ and } r\in (0,r_1).
 \end{eqnarray}
 Let $f\in C^{\alpha}(\partial D)$ be such that
\begin{eqnarray}
|f(x)-f(y)| \leq E |x-y|^{\alpha} \ \text{for every } x, y \in \partial D.
\end{eqnarray}
If
\begin{eqnarray}\label{L1pesato}
\int_{\partial D} |f|w \leq \varepsilon \
 \end{eqnarray}
 then
 \begin{eqnarray}\label{stimalog}
 \left \Vert f\right \Vert _{L^{\infty }\left( \partial D\right) }\leq
E \eta\left(\frac{\varepsilon}{E}\right)
\end{eqnarray}
 where $\eta$ satisfies \eqref{eta} with constants only depending on $M,K, r_0, \alpha$.
 \end{proposition}

 \begin{proof}
 By \eqref{infinito} we have trivially
 \begin{eqnarray}\label{zeriexp1}
 \int_{\Delta_r(x)}w \geq M^{-1} \exp(- 2Kr^{-K}) \ \text{for every } x \in \partial D\text{ and } r\in (0,r_1).
 \end{eqnarray}
 Now fix $\xi \in \partial D$ such that $|f(\xi)|=  \|f\|_{L^{\infty}(\partial D)}$. For every $r>0$ and $x \in \Delta_r(\xi)$ we have
 \begin{eqnarray}\label{holder}
 |f(\xi)|\leq |f(x)| + Er^{\alpha} \ ,
 \end{eqnarray}
 by multiplying both sides of \eqref{holder} by $w$ and integrating over  $\Delta_r(\xi)$ we get
\begin{eqnarray}\label{integralino}
|f(\xi)|\int_{\Delta_r(\xi)}w  \leq \int_{\Delta_r(\xi)}w |f| + Er^{\alpha} \int_{\Delta_r(\xi)}w  \ ,
 \end{eqnarray}
 hence
 \begin{eqnarray}\label{integralino2}
\|f\|_{L^{\infty}(\partial D)}\leq \frac{\varepsilon}{\int_{\Delta_r(\xi)}w} +  Er^{\alpha} \leq \nonumber \\ \leq M \exp( 2Kr^{-K})\varepsilon +  Er^{\alpha}.
\end{eqnarray}
Now, if  $\varepsilon <E$ and $r=(4K)^{1/K} |\log \frac{\varepsilon}{E} |  ^{- 1/K }< r_1$, then we fix this value of $r$ in \eqref{integralino2} and the thesis follows. The remaining cases are trivial.
\end{proof}

\section{Conclusion} \label{final}

\begin{proof}[Proof of Theorem \ref{stabl}.]

By a standard interpolation estimate we deduce that
\begin{eqnarray}\label{int}
\|u_1(\lambda_1-\lambda_2)\|_{L^2(\partial D)}\le C \|u_1(\lambda_1-\lambda_2)\|^{\frac{1}{3}}_{H^1(\partial D)} \|u_1(\lambda_1-\lambda_2)\|^{\frac{2}{3}}_{H^{-\frac{1}{2}}(\partial D)}
\end{eqnarray}
where $C>0$ only depends on the a priori data.

It is an elementary observation, that we have
\begin{eqnarray}
\|u_1(\lambda_1-\lambda_2)\|_{H^1(\partial D)}\le\|\lambda_1-\lambda_2 \|_{C^{0,1}(\partial D)}\|u_1\|_{H^1(\partial D)}\le C
\end{eqnarray}
where $C>0$ is a constant depending on the a priori data only.

On the other hand, by the impedance condition
\begin{eqnarray}
\dfrac{\partial u_i}{\partial \nu} + i\lambda_i u_i= 0 \ \ \ \mbox{on} \ \partial D \ ,
\end{eqnarray}

for $i=1,2$ we also have that

\begin{eqnarray}
\|u_1(\lambda_1-\lambda_2)\|_{H^{-\frac{1}{2}}(\partial D)}\le \left \|\dfrac{\partial u_1}{\partial \nu} -\dfrac{\partial u_2}{\partial \nu}  \right\|_{H^{-\frac{1}{2}}(\partial D)} + C\|u_1 -u_2\|_{H^{-\frac{1}{2}}(\partial D)} \ .
\end{eqnarray}

By combining the results in Theorem \ref{stabpc} and in Corollary \ref{hmenounmezzo} we get
\begin{eqnarray}
\|u_1(\lambda_1-\lambda_2)\|_{H^{-\frac{1}{2}}(\partial D)}\le \eta(\varepsilon)\ .
\end{eqnarray}
where $C>0$ is a constant depending on the a priori data only.

Hence, by the inequality \eqref{int} we arrive at
\begin{eqnarray}
\|u_1(\lambda_1-\lambda_2)\|_{L^2(\partial D)}\le \eta(\varepsilon)\ .
\end{eqnarray}
Now, we apply Proposition \ref{weighted} with $w=|u_1|^2$ and $f=(\lambda_1-\lambda_2)^2$ and the conclusion follows.
\end{proof}

\end{document}